\documentclass[11pt]{article}
\usepackage[T1]{fontenc}
\usepackage{graphicx,amsmath,amsthm,amsfonts,amssymb,microtype,thmtools,underscore,mathtools,wrapfig,anyfontsize,thm-restate}
\usepackage[shortlabels]{enumitem}
\setlist[itemize]{topsep=0ex,itemsep=0ex,parsep=0.4ex}
\setlist[enumerate]{topsep=0ex,itemsep=0ex,parsep=0.4ex}
\usepackage[dvipsnames,svgnames,table]{xcolor}
\usepackage[unicode=true]{hyperref}
\hypersetup{
colorlinks=true,
breaklinks=true,
linkcolor={black},
citecolor={black},
urlcolor={blue!60!black},
pdftitle={Polynomial Bounds in the Apex Minor Theorem}}
\usepackage[noabbrev,capitalise]{cleveref}
\crefname{lem}{Lemma}{Lemmas}
\crefname{thm}{Theorem}{Theorems}
\crefname{cor}{Corollary}{Corollaries}
\crefname{prop}{Proposition}{Propositions}
\crefname{conj}{Conjecture}{Conjectures}
\crefname{open}{Open Problem}{Open Problems}
\crefname{claim}{Claim}{Claims}
\crefformat{equation}{(#2#1#3)}
\Crefformat{equation}{Equation #2(#1)#3}

\newcommand{\defn}[1]{\textcolor{Maroon}{\emph{#1}}}

\newcommand{\GG}{\mathcal{G}}

\usepackage[longnamesfirst,numbers,sort&compress]{natbib}
\makeatletter
\def\NAT@spacechar{~}
\makeatother
\usepackage[margin=32mm]{geometry}
\DeclarePairedDelimiter{\floor}{\lfloor}{\rfloor}
\DeclarePairedDelimiter{\ceil}{\lceil}{\rceil}

\renewcommand{\geq}{\geqslant}
\renewcommand{\leq}{\leqslant}

\DeclareMathOperator{\dist}{dist}
\DeclareMathOperator{\rad}{rad}
\DeclareMathOperator{\diam}{diam}

\DeclareMathOperator{\tw}{tw}
\DeclareMathOperator{\ttw}{tree-tw}

\DeclareMathOperator{\pw}{pw}

\renewcommand{\thefootnote}{\fnsymbol{footnote}}
\theoremstyle{plain}
\newtheorem{thm}{Theorem}
\newtheorem{lem}[thm]{Lemma}
\newtheorem{cor}[thm]{Corollary}
\newtheorem{prop}[thm]{Proposition}

\theoremstyle{definition}
\newtheorem{conj}[thm]{Conjecture}

\begin{document}
\title{\bf Polynomial Bounds in the Apex Minor Theorem}
\author{Kevin Hendrey\footnotemark[2]
\qquad David R. Wood\footnotemark[2]}
\maketitle
\footnotetext[2]{School of Mathematics, Monash   University, Melbourne, Australia\\ 
\hspace*{4.5mm} \texttt{\{Kevin.Hendrey1,David.Wood\}@monash.edu}. }
\begin{abstract}
A graph $A$ is \defn{apex} if $A-z$ is planar for some vertex $z\in V(A)$. Eppstein [\emph{Algorithmica}, 2000] showed that for a minor-closed class $\GG$, the graphs in $\GG$ with bounded radius have bounded treewidth if and only if some apex graph is not in $\GG$. In particular, for every apex graph $A$ and integer $r$, there is a minimum integer $g(A,r)$ such that every $A$-minor-free graph with radius $r$ has treewidth at most $g(A,r)$. We show that if $t=|V(A)|$ then $g(A,r)\in O^\ast(r^9t^{18})$ which is the first upper bound on $g(A,r)$ with polynomial dependence on both $r$ and $t$. More precisely, we show that every $A$-minor-free graph with radius $r$ has no $16rt^2 \times 16rt^2$ grid minor, which implies the first result via the Polynomial Grid Minor Theorem. A key example of an apex graph is the complete bipartite graph $K_{3,t}$, since $K_{3,t}$-minor-free graphs include and generalise graphs embeddable in any fixed surface. In this case, we prove that every $K_{3,t}$-minor-free graph with radius $r$ has no $4r(1+\sqrt{t})\times 4r(1+\sqrt{t})$ grid minor, which is tight up to a constant factor.
\end{abstract}
\renewcommand*{\thefootnote}{\arabic{footnote}}

\section{Introduction}
\label{Intro}

Treewidth\footnote{We consider finite simple undirected graphs $G$ with vertex-set $V(G)$ and edge-set $E(G)$. 
A graph $H$ is a \defn{minor} of a graph $G$ if a graph isomorphic to $H$ can be obtained from a subgraph of $G$ by contracting edges. A graph $G$ is \defn{$H$-minor-free} if $H$ is not a minor of $G$. A graph class $\mathcal{G}$ is \defn{minor-closed} if for every $G\in\GG$ every minor of $G$ is in $\GG$. For a graph $H$, an \defn{$H$-model} in a graph $G$ is a collection $\{B_u:u\in V(H))$ of pairwise disjoint connected subgraphs of $G$, such that for each edge $uv\in V(H)$ there is an edge of $G$ between $B_u$ and $B_v$ which is said to \defn{represent} $uv$. Note that $H$ is a minor of $G$ if and only if there is an $H$-model in $G$.  A \defn{tree-decomposition} of a graph $G$ consists of a tree $T$ and a collection $(B_x:x \in V(T))$ such that: (a) $\bigcup\{B_x:x\in V(T)\}=V(G)$, (b) for every edge ${vw \in E(G)}$, there exists a node ${x \in V(T)}$ with ${v,w \in B_x}$; and 	(b) for every vertex ${v \in V(G)}$, the set $\{ x \in V(T) : v \in B_x \}$ induces a non-empty (connected) subtree of $T$. The \defn{width} of $(B_x:x \in V(T))$ is ${\max\{ |B_x| : x \in V(T) \}-1}$. The \defn{treewidth} of a graph $G$, denoted \defn{$\tw(G)$}, is the minimum width of a tree-decomposition of $G$. A \defn{path-decomposition} is a tree-decomposition where the underlying tree is a path, denoted by the corresponding sequence of bags. The \defn{pathwidth} of a graph $G$, denoted \defn{$\pw(G)$}, is the minimum width of a path-decomposition of $G$. } is the standard measure of how similar a graph  is to a tree, and is an important parameter in structural graph theory, especially Robertson and Seymour's graph minor theory, and also in algorithmic graph theory, since many NP-complete problems are solvable in linear time on graphs with bounded treewidth. See \citep{HW17,Bodlaender98,Reed97} for surveys on treewidth. 

A key example in the study of treewidth is the $n\times n$ grid \footnote{An \defn{$n\times m$ grid} is any graph isomorphic to the graph with vertex-set $\{1,\dots,n\}\times\{1,\dots,m\}$, where $(x_1,y_1)$ is adjacent to $(x_2,y_2)$ if and only if $|x_1-x_2|+|y_1-y_2|=1$.}, which has treewidth $n$ (see \citep{HW17}). In fact, a graph has large treewidth if and only if it has a large grid minor, as shown by the following `Grid Minor Theorem' of \citet{RS-V}, which is a cornerstone of graph minor theory.

\begin{thm}[\citep{RS-V}]
\label{GridMinorTheorem}
For every integer $k\geq 1$ there is a minimum integer $g(k)$ such that every graph with no $k\times k$ grid minor has treewidth at most $g(k)$.  
\end{thm}

Given the importance of \cref{GridMinorTheorem},  there has been substantial effort towards finding simplified proofs with improved bounds on $g(k)$; see \citep{DJGT-JCTB99,LeafSeymour15,RST94,KK20,CC16,Chuzhoy15,DH08,DHK-Algo09}. The best known bound, due to \citet{CT21}, says that 
\begin{equation}
    \label{GMT}
    g(k)\in O^\ast(k^9),
\end{equation}
where $O^\ast(\ldots)$ notation hides poly-logarithmic factors. 

The next lemma is an easy result of \citet[(1.3) \& (1.4)]{RST94}.

\begin{lem}[\citep{RST94}]
\label{PlanarGrid}
Every planar graph with $t$ vertices is a minor of the $2t\times2t$ grid. 
\end{lem}

\cref{GridMinorTheorem,PlanarGrid} imply the following `Planar Minor Theorem'.

\begin{thm}[\citep{RS-V,RST94}]
\label{PlanarMinorTheorem}
For every planar graph $H$ there is a minimum integer $g(H)$ such that every $H$-minor-free graph has treewidth at most $g(H)$. 
\end{thm}

A result like \cref{PlanarMinorTheorem} only holds for planar graphs, since for any non-planar graph $H$,  grids are $H$-minor-free with unbounded treewidth. Put another way, a minor-closed class $\GG$ has bounded treewidth if and only if some planar graph is not in $\GG$. 

The following natural question arises: do more general minor-closed classes have bounded treewidth under some extra assumptions. Given that grids have unbounded radius, it is natural to consider bounded radius as the extra assumption. This direction was initiated by \citet{Eppstein-Algo00}\footnote{Eppstein worked with diameter instead of radius, but the difference is negligible, since $\rad(G) \leq \diam(G) \leq 2 \rad(G)$ for every graph $G$.} who proved the following `Apex Minor Theorem'. 

Here a graph $A$ is \defn{apex} if $A-z$ is planar for some vertex $z$ in $A$ (or $V(A)=\emptyset$). Apex-minor-free graphs include planar graphs (since $K_5$ or $K_{3,3}$ is apex), and more generally, includes graphs embeddable in any fixed surface\footnote{The \defn{Euler genus} of a surface with $h$ handles and $c$ crosscaps is $2h+c$. The \defn{Euler genus} of a graph $G$ is the minimum Euler genus of a surface in which $G$ embeds with no crossings. It follows from Euler's formula that every $n$-vertex $m$-edge graph $G$ with Euler genus at most $g$ satisfies $m\leq 3(n+g-2)$, and if $G$ is bipartite then $m\leq 2(n+g-2)$. In particular, if $G=K_{3,t}$ then $3t=m\leq 2(n+g-2)=2(t+g+1)$, implying $t\leq 2g+2$. That is, $K_{3,2g+3}$ has Euler genus greater than $g$. Since Euler genus is minor-monotone, every graph with Euler genus at most $g$ is $K_{3,2g+3}$-minor-free.} (since it follows from Euler's formula that graphs of Euler genus $g$ are $K_{3,2g+3}$-minor-free, and $K_{3,t}$ is apex since $K_{2,t}$ is planar). 

\begin{thm}[\citep{Eppstein-Algo00}]
\label{ApexMinorTheorem}
For every apex graph $A$ and integer $r\geq 0$ there is a minimum integer $g(A,r)$ such that every $A$-minor-free graph with radius at most $r$ has treewidth at most $g(A,r)$. 
\end{thm}

A result like \cref{ApexMinorTheorem} only holds for apex graphs, since for any non-apex graph $H$, the graphs obtained from grids by adding one dominant vertex are $H$-minor-free with radius $1$ and unbounded treewidth. That is, for a minor-closed class $\GG$, the graphs in $\GG$ with bounded radius have bounded treewidth if and only if some apex graph is not in $\GG$. 

Eppstein's motivation for proving  \cref{ApexMinorTheorem} is that it leads to efficient approximation algorithms for apex-minor-free graphs, since Baker's method~\citep{Baker94}, originally developed for planar graphs, works in any minor-closed class where graphs of bounded radius have bounded treewidth. See \citep{DH08a,DFHT04,DHT06} for more algorithmic applications in this direction. 

As in the Grid Minor Theorem, there has been substantial interest in obtaining simpler proofs and improved bounds in the Apex Minor Theorem. First note that the bound on $g(A,r)$ in Eppstein's proof of the Apex Minor Theorem is at least exponential in $r$. \citet{DH-Algo04} gave a simpler proof, still with exponential dependence on $r$. Both these proofs use the Grid Minor Theorem, and the obtained bounds on $g(A,r)$ are still exponential in $r$ even using the above-mentioned polynomial bound in the Grid Minor Theorem. 

\citet{DH08} gave another proof of the Apex Minor Theorem with linear dependence on $r$. Such a result can also be concluded from recent work on layered treewidth~\citep{DMW17} or row treewidth~\citep{DJMMUW20}. All these proofs use the Graph Minor Structure Theorem of \citet{RS-XVI} (which in turn uses the Grid Minor Theorem), and as a consequence the dependence on $|V(A)|$ is very large. 

In summary, all the known proofs of the Apex Minor Theorem have at least exponential dependence on $r$ or $|V(A)|$. The goal of this paper is to address this shortcoming. 

The first contribution of this paper, presented in \cref{SimpleProof},  is to give a new simple proof of the Apex Minor Theorem. The main result of this paper, presented in \cref{ProbProof}, gives a more involved proof of the Apex Minor Theorem with polynomial dependence on both $|V(A)|$ and $r$. Like all previous proofs of the Apex Minor Theorem, both our proofs use the Grid Minor Theorem, so our focus is entirely on grid minors in apex-minor-free graphs. In particular, we prove the following result:

\begin{thm}
\label{PolyApexMinorTheorem}
Let $A$ be an apex graph such that $A-z$ is planar for some $z\in V(A)$. Let $t:=|V(A)|$ and $d:=\deg_A(z)$. Then any $A$-minor-free graph $G$ with radius at most $r$ has no $16 rtd \times 16rtd$ grid minor, implying $\tw(G)\in O^\ast(r^9t^9d^9) \subseteq O^\ast(r^9t^{18})$.
\end{thm}

This is the first such bound on $\tw(G)$ that is polynomial in both $r$ and $|V(A)|$. 

We now show that \cref{PolyApexMinorTheorem} can be combined with a result of \citet{KK20} to prove the Apex Minor Theorem with linear dependence on $r$ and explicit dependence on $|V(A)|$. (All previous proofs of the Apex Minor Theorem with linear dependence on $r$ use Robertson and Seymour's Graph Minor Structure Theorem, which has much larger dependence on $|V(A)|$.)

\begin{cor}
\label{LinearApexMinorTheorem}
For any apex graph $A$ with $t$ vertices, every $A$-minor-free graph $G$ with radius at most $r$ has treewidth at most $t^{O(t)} r$.
\end{cor}

\begin{proof}
\citet{KK20} proved the following version of the Grid Minor Theorem (without using the Graph Minor Structure Theorem): For any graph $H$, every $H$-minor-free graph with no $k\times k$ grid minor has treewidth at most $|V(H)|^{O(|E(H)|)}k$. 
By \cref{PolyApexMinorTheorem}, this result is applicable with $H=A$ and $|V(H)|=t$ and $|E(H)|<4t$ and $k=16rt^2$. Thus $\tw(G) \leq t^{O(t)} (16 rt^2)\leq t^{O(t)}r$. 
\end{proof}

Our next contribution is to obtain improved bounds for specific graphs $A$ in the Apex Minor Theorem. Note that several papers have followed a similar direction for the Planar Minor Theorem\footnote{\citet{BRST91} showed that if $H$ is a forest, then $g(H)=|V(H)|-2$, and in fact every $H$-minor-free graph has pathwidth at most $|V(H)|-2$. \citet{FL94} showed that if $H$ is a cycle then $g(H)=|V(H)|-2$. \citet{BBR-DAM09} showed that $g(3)\leq 7$. If $H$ is an apex-forest, then \citet{LeafSeymour15} showed that $g(H)\leq \frac32|V(H)|-2$. In particular, $g(K_{2,t}) \leq \frac32 t+1$. If $H$ is a wheel, then \citet{RT17a} showed that $g(H)\leq 36|V(H)|-39$, improved to $2|V(H)|+O(\sqrt{|V(H)|})$ by \citet{GHOR24}. \citet{GHOR24} also showed  that $g(4)\leq 160$, and that if $H$ is a disjoint union of $k$ cycles, then $g(H)\leq \frac32|V(H)|+O(k^2\log k)$ and $g(H)\leq|V(H)|+O(\sqrt{|V(H)|})$ if $k=2$.}. We focus on $K_{3,t}$-minor-free graphs, which include and generalise graphs embeddable in any fixed surface, as explained above. 

\begin{restatable}{thm}{blah}
\label{K3tMinorFreeIntro}
Every $K_{3,t}$-minor-free graph $G$ with radius at most $r$ has  no $n\times n$ grid minor with $n:=\lceil 4r(1+\sqrt{t-1})\rceil$, implying $\tw(G)\in O^\ast(r^9t^{9/2})$.
\end{restatable}

\cref{K3tMinorFreeIntro} implies the following bound on grid minors in graphs embeddable in a fixed surface\footnote{The best previous bound in \cref{SurfaceIntro} that we are aware of can be concluded from a result of \citet{Eppstein-Algo00}, which says that every graph with Euler genus $g$ and radius $r$ has treewidth $O(gr)$, which implies no $O(gr)\times O(gr)$ grid minor. This can also be concluded from results about layered treewidth~\citep{DMW17} or row treewidth~\citep{DJMMUW20}.}:

\begin{cor}
\label{SurfaceIntro}
Every graph with Euler genus $g$ and radius $r$ has  no 
$n\times n$ grid minor with $n:= \lceil4r(1+\sqrt{2g+2})\rceil$.
\end{cor}

Our next result shows that \cref{K3tMinorFreeIntro,SurfaceIntro} are best possible up to a constant factor.

\begin{prop}
\label{LowerBoundIntro}
For any integers $g\geq 2$ and $r\geq 1$ there is a graph $G$ with radius at most $r$ and Euler genus at most $g$ (implying $G$ is $K_{3,2g+3}$-minor-free), such that the $n\times n$ grid is a subgraph of $G$, where $n:=(2r-1)\floor{\sqrt{g/2}} \approx r\sqrt{2g}$.
\end{prop}

In \cref{ttw}, we give an application of \cref{PolyApexMinorTheorem,K3tMinorFreeIntro} to tree-treewidth, which is a parameter introduced
recently by \citet{LNW}.

\section{Simple Proof}
\label{SimpleProof}

This section presents a simple proof of the Apex Minor Theorem. The next lemma is the key. A \defn{centre} of a graph $G$ is a vertex $\alpha\in V(G)$ such that $\max\{\dist(\alpha,v):v\in V(G)\}$ is equal to the radius of $G$.

\begin{lem}
\label{key}
Let $A$ be a planar graph such that $A-z$ is a minor of the $k\times k$ grid for some $z\in V(A)$. Let $G$ be an $A$-minor-free graph with radius $r$ and centre $\alpha$. Then $G-\alpha$ has no $k^r\times k^r$ grid minor. 
\end{lem}

\begin{proof}
We proceed by induction on $r$. In the $r=0$ case, $V(G-\alpha)=\emptyset$ and the result holds. Now consider the radius $r$ case, and assume the radius $r-1$ case holds. Let $V_i:=\{x\in V(G):\dist_G(x,\alpha)=i\}$ for $i\in\{0,\dots,r\}$. So $V_0=\{\alpha\}$ and $V_0\cup\dots\cup V_r=V(G)$. 

Suppose for the sake of contradiction that $G-\alpha$ has a $k^r\times k^r$ grid minor. Partition this $k^r\times k^r$ grid into $k\times k$ subgrids, so that contracting each subgrid to a vertex gives a $k^{r-1}\times k^{r-1}$ grid. 

First suppose that some $k\times k$ subgrid is contained in $V_r$. Thus $H$ is a minor of $G[V_r]$. By construction, $G[V_0\cup \dots\cup V_{r-1}]$ is connected. Let $G'$ be obtained from $G$ by contracting $V_0\cup \dots\cup V_{r-1}$ to a vertex $w$. Since every vertex in $V_r$ has a neighbour in $V_{r-1}$,  every vertex in $G'$ is at distance at most $1$ from $w$.  Since $H$ is a minor of  $G[V_r]$, $A$ is a minor of $G'$ and thus of $G$, which is a contradiction. 

Now assume that every $k\times k$ subgrid intersects $V_1\cup\dots\cup V_{r-1}$. Let $G'$ be obtained from $G$ by contracting each subgrid to a vertex, and deleting any vertices in $V_r$ not in one of the subgrids. So every vertex in $G'$ is at distance at most $r-1$ from $\alpha$, and $G'-\alpha$ contains a $k^{r-1}\times k^{r-1}$ grid minor. Since $G'$ is a minor of $G$, $G'$ is $A$-minor-free. Hence $G'$ contradicts the assumed truth of the $r-1$ case. 

Therefore,   $G-\alpha$ has no $k^r\times k^r$ grid minor.
\end{proof}

\cref{PlanarGrid,key} imply that for every apex graph $A$ with $t$ vertices, every $A$-minor-free graph $G$ with radius $r$ has no $(2t-2)^{r}\times (2t-2)^{r}$ grid minor. Equation~\cref{GMT} implies that $\tw(G) \in O^\ast( (2t)^{9r})$. This completes our simple proof of the Apex Minor Theorem by \citet{Eppstein-Algo00} (\cref{ApexMinorTheorem}).

\section{Polynomial Upper Bound}
\label{ProbProof}

This section proves \cref{PolyApexMinorTheorem} from \cref{Intro}. We need the following straightforward lemma.

\begin{lem}
\label{DoublingTrick}
If a graph $G$ is a minor of the $k\times \ell$ grid, then there is an $H$-model $(B_u:u\in V(H))$ in the $2k\times 2\ell$ grid $J$, such that for each vertex $u\in V(H)$ there is a vertex $h_u$ in $B_u$ incident to no edge of $J$ representing an edge of $H$. 
\end{lem}

\begin{proof}
By assumption, there is an $H$-model $\phi=(X_u:u\in V(H))$ in the $k\times \ell$ grid. We now define 
an $H$-model $\phi'=(B_u:u\in V(H))$ in the $2k\times 2\ell$ grid. For each $u\in V(H)$, let $B_u:=\{ (2x,2y),(2x-1,2y),(2x,2y-1),(2x-1,2y-1): (x,y)\in V(X_u) \}$. For distinct $u,v\in V(H)$, since $X_u$ and $X_v$ are disjoint, $B_u$ and $B_v$ are disjoint. Consider an edge $uv\in E(H)$. If $uv$ is represented by a horizontal edge $(x,y)(x+1,y)$ with respect to $\phi$, then represent $uv$ by the edge $(2x,2y)(2x+1,2y)$ with respect to $\phi'$. If $uv$ is represented by a vertical edge $(x,y)(x,y+1)$ with respect to $\phi$, then represent $uv$ by the edge $(2x,2y)(2x,2y+1)$ with respect to $\phi'$. For each vertex $u\in V(H)$, choose one vertex $(x,y)$ in $X_u$ and let $h_u:=(2x-1,2y-1)$, which is in $B_u$. Every vertex of $J$ that is incident to a representing edge has at least one even coordinate, so $h_u$ is incident to no edge of $J$ representing an edge of $H$.
\end{proof}

The next lemma is the heart of the proof of \cref{PolyApexMinorTheorem}.

\begin{lem}
\label{ApexGridMinor}
Let $A$ be an apex graph, where $H:=A-z$ is planar for some $z\in V(A)$ with $d:=\deg_A(z)$. Assume that $H$ is a minor of the $k\times \ell$ grid. Let $G$ be an $A$-minor-free graph with radius $r$. Then $G$ has no $(2kn+1)\times(2\ell n+1)$ grid minor, where $n:=4(r-1)d+1$. 
\end{lem}

\begin{proof}
By \cref{DoublingTrick}, there is an $H$-model $(B_u:u\in V(H))$ in the $2k\times 2\ell$ grid $J$, such that for each vertex $u\in V(H)$ there is a vertex $h_u$ in $B_u$ incident to no edge of $J$ representing an edge of $H$. 

Let $\alpha$ be a centre of $G$. 
Suppose for the sake of contradiction that $G$ has a $(2kn+1)\times(2\ell n+1)$ grid minor. 
So $G-\alpha$ has a $2kn\times 2\ell n$ grid model $\{X_{i,j}: 1\leq i\leq 2kn, 1\leq j\leq 2\ell n\}$. 
Let $G'$ be obtained from $G$ by contracting $X_{i,j}$ to a vertex denoted $(i,j)$ (for all $i,j$). 
So the $2kn\times 2\ell n$ grid is a subgraph $L$ of $G'$ with vertex-set $\{1,\dots,2kn\}\times\{1,\dots,2\ell n\}$. 
	
Since contractions do not increase distances, $G'$ has radius at most $r$ with centre $\alpha$. 
So for each vertex $x\in V(G')$ there is a path $P_x$ in $G'$ from $x$ to $\alpha$ with at most $r-1$ internal vertices. Let $I_x$ be the set of internal vertices of $P_x$, which may be empty if $x\alpha\in E(G')$. 
	
We now partition $L$ into $n\times n$ subgrids, so that contracting each subgrid to a vertex gives a $2k \times 2\ell$ grid. Fix $a\in\{1,\dots,2k\}$ and $b\in\{1,\dots,2\ell\}$. For $p,q\in\{1,\dots,n\}$, 
let $(a,b,p,q)$ denote the vertex $((a-1)n+p,(b-1)n+q)$. So $\{(a,b,p,q): p,q\in\{1,\dots,n\} \}$ induces an $n \times n$ subgrid \defn{$S_{a,b}$}, said to \defn{belong} to the vertex $u\in V(H)$ such that $(a,b)\in V(B_u)$ (if such a $u$ exists). For $i\in\{1,\dots,n\}$, let $X^{\rightarrow}_{a,b,i}:=\{ (a,b,p,i): p\in\{1,\dots,n\}\}$, called a \defn{horizontal bar} in $S_{a,b}$; let $X^{\uparrow}_{a,b,i}:=\{ (a,b,i,p): p\in\{1,\dots,n\}\}$, called a \defn{vertical bar} in $S_{a,b}$; and let $X_{a,b,i}:=X^{\rightarrow}_{a,b,i} \cup X^{\uparrow}_{a,b,i}$, called a \defn{cross} in $S_{a,b}$.  
Note that any horizontal bar in $S_{a,b}$ intersects any vertical bar in $S_{a,b}$.

Let $A$ be the set of all pairs $(v,x)$ such that for some $u\in N_A(z)$ with $h_u=(a,b)$ and for some $i\in\{1,\dots,n\}$, we have $v=(a,b,i,i)$ and $x\in I_v$. Note that $|A|\leq dn(r-1)$. 

For  $a\in\{1,\dots,2k\}$ and $b\in\{1,\dots,2\ell\}$, let $m_{a,b}$ be a uniformly random element of $\{1,\dots,n\}$, chosen independently.
For $u\in V(H)$, if $h_u=(a,b)$ then let $h'_u$ be the vertex $(a,b,m_{a,b},m_{a,b})$, which is in $X_{a,b,m_{a,b}}$.
As illustrated in \cref{FindApex}, for $u\in V(H)$, let 
$C_u$ be the subgraph of $G'$ induced by the union of the following sets: 
\begin{align*}
     \{ X_{a,b,m_{a,b}} &: (a,b) \in V(B_u)\} \\
 \{  X^{\rightarrow}_{a,b,m_{a+1,b}}  &: (a,b)(a+1,b) \in E(B_u) \}\\
 \{  X^{\uparrow}_{a,b,m_{a,b+1}}  &: (a,b)(a,b+1) \in E(B_u) \}\\
 \{  X^{\rightarrow}_{a,b,m_{a+1,b}}  &: (a,b)(a+1,b)\text{ represents some edge }uv\text{ incident to }u\text{ with }(a+1,b)\in B_v\}\\ 
\{  X^{\uparrow}_{a,b,m_{a,b+1}}  &: (a,b)(a,b+1)\text{ represents some edge }uv\text{ incident to }u\text{ with }(a,b+1)\in B_v\}.
\end{align*}

\begin{figure}
\includegraphics{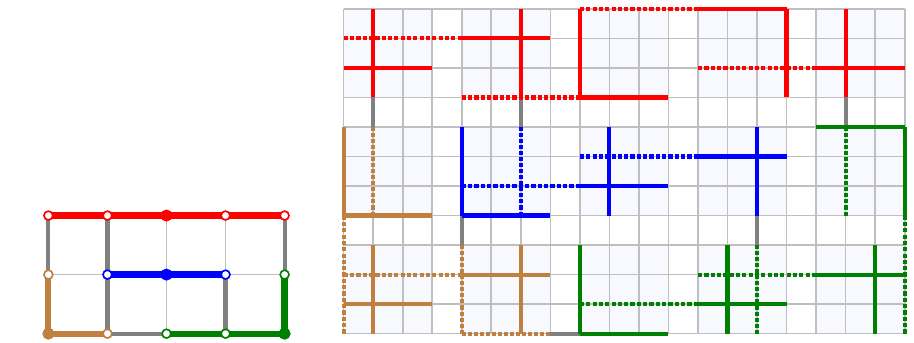}
\caption{Example with $H=K_4$: (left) the model $\{B_u:u\in V(H)\}$ in a grid, highlighting the vertex $h_u$, (right) 
the model $\{C_u:u\in V(H)\}$ in a larger grid, where dashed edges are bars arising from neighbouring blocks. }
\label{FindApex}
\end{figure}

We now show that $C_u$ is connected for each vertex $u\in V(H)$. For each $(a,b)\in V(B_u)$, $G'[V(C_u)\cap S_{a,b}]$ is connected, since it consists of one cross plus some number of vertical or horizontal bars. For an edge $(a,b)(a+1,b)\in E(B_u)$, the vertex $(a,b,n,m_{a+1,b})\in S_{a,b}$ is adjacent to the vertex $(a+1,b,1,m_{a+1,b})\in S_{a+1,b}$. Similarly, for an  edge $(a,b)(a,b+1)\in E(B_u)$, the vertex $(a,b,m_{a+1,b},n)\in S_{a,b}$ is adjacent to the vertex $(a,b+1,m_{a+1,b},1)\in S_{a,b+1}$. Since $B_u$ is connected, so is $C_u$. Every subgrid that $C_u$ intersects belongs to $u$. So $C_u $ and $C_v$ are disjoint for distinct $u,v\in V(H)$. We now show that 
$C_u$ is adjacent to $C_v$ for each edge $uv\in E(H)$. If $(a,b)(a+1,b)$ is the edge of $J$ representing $uv$, then  $(a,b,n,m_{a+1,b})\in V(C_u)$ is adjacent to  $(a+1,b,1,m_{a+1,b})\in V(C_v)$. Similarly, if $(a,b)(a,b+1)$ is the edge of $J$ representing $uv$, then  $(a,b,m_{a,b+1})\in V(C_u)$ is adjacent to  $(a,b+1,m_{a,b+1},1)\in V(C_v)$. Hence, $C_u$ is adjacent to $C_v$. Therefore, $(C_u:u\in V(H))$ is an $H$-model in $L$. 

Consider a vertex $x=(a,b,p,q) \in V(L)$, and consider the probability that $x$ is in $\bigcup\{C_u:u\in V(H)\}$. 
For this to happen, it must be that $(a,b)\in V(B_u)$ for some $u\in V(H)$, and there is only one such vertex $u$. 
Moreover, if $x\in V(C_u)$ then $x\in 
X^{\rightarrow}_{a,b,m_{a,b}} \cup 
X^{\uparrow}_{a,b,m_{a,b}} \cup 
X^{\rightarrow}_{a,b,m_{a+1,b}} \cup
X^{\uparrow}_{a,b,m_{a,b+1}}$. 
By construction, 
$\mathbb{P}(x\in X^{\rightarrow}_{a,b,m_{a,b}})=\mathbb{P}(x\in X^{\uparrow}_{a,b,m_{a,b}})=\mathbb{P}(x\in X^{\rightarrow}_{a,b,m_{a+1,b}})=\mathbb{P}(x\in X^{\uparrow}_{a,b,m_{a,b+1}})=\frac{1}{n}$. 
By the union bound, $x$ is in $V(C_u)$, and thus in 
$\bigcup\{C_u:u\in V(H)\}$,  with probability at most $\frac{4}{n}$. 

Consider a pair $(v,x)\in A$. So 
$v=(a,b,i,i)$ and $x\in I_v$, for some $u\in N_A(z)$ with $h_u=(a,b)$ and for some $i\in\{1,\dots,n\}$. Say $(v,x)$ is \defn{bad} if $i=m_{a,b}$ and $x\in C_w$ for some vertex $w\in V(H-u)$. 
The probability that $i=m_{a,b}$ equals $\frac1n$. 
The probability that $x\in C_w$ for some vertex $w\in V(H-u)$ is at most $\frac{4}{n}$. 
Moreover, these events are independent since $w\neq u$ and $h_u$ is incident to no edge of $J$ representing an edge of $H$ incident to $u$. In particular, $uw$ is not an edge of $H$ represented by the edge of $L$ between $(a,b)$
and $(a',b')$ where $x\in S_{a',b'}$. 
Thus, the probability that $(v,x)$ is bad is at most $\frac{4}{n^2}$. 
By the union bound, the probability that some pair in $A$ is bad is at most $\frac{4|A|}{n^2}\leq \frac{4(r-1)dn}{n^2} = \frac{4(r-1)d}{n}<1$. 
Hence, with positive probability, no pair in $A$ is bad. 
Therefore, there exists $(m_{a,b}: a\in\{1,\dots,2k\},b\in\{1,\dots,2\ell\})$ such that no pair in $A$ is bad. Let $(C_u:u\in V(H))$ be the $H$-model in $L$ defined by  $(m_{a,b}:
 a\in\{1,\dots,2k\}, b\in\{1,\dots,2\ell\})$.
 
For each $u\in N_A(z)$, let $P'_u$ be the path of $P_{h'_u}-h'_u$, and note that $P'_u$ contains $\alpha$ and some neighbour of $h'_u$. Since no pair in $A$ is bad, $P'_u$ avoids $\bigcup\{C_w:w\in V(H)\}$. 
Thus we can extend $(C_u:u\in V(H))$ to a model of $A$ in $G'$ by setting $C_z:=\bigcup\{P'_u:u\in N_A(z)\}$, which is the desired contradiction. 
\end{proof}

\cref{PlanarGrid} says that \cref{ApexGridMinor} is applicable with $k=\ell\leq 2|V(A)|-2$, which with \eqref{GMT} implies \cref{PolyApexMinorTheorem}.

\section{\boldmath $K_{3,t}$-Minor-Free Graphs}

This section proves upper bounds on the size of grid minors in  $K_{3,t}$-minor-free graphs with given radius. 
It is easily seen that $K_{2,t}$ is a minor of a $O(\sqrt{t})\times O(\sqrt{t})$ grid. So by 
\cref{ApexGridMinor}, every $K_{3,t}$-minor-free graph with radius $r$ has no $O(rt^{3/2}) \times O(rt^{3/2})$ grid minor. We can do better, as follows.

\begin{lem}
\label{K3tGridRadius}
Let $G$ be a graph with radius $r\geq 1$ and centre $\alpha$. If $G-\alpha$ has an $n\times m$ grid minor, then $K_{3,t}$ is a minor of $G$ for some
$$t\geq \frac{(n-4r+2)(m-4r+2)}{8r(2r-1)} .$$
\end{lem}

\begin{proof}
Let   $(X_{i,j}:1\leq i\leq n, 1\leq j \leq m)$ be a model of the $n\times m$ grid minor in $G-\alpha$. 
Let $G'$ be the graph obtained from $G$ by contracting $X_{i,j}$ to a vertex $(i,j)$ (for all $i,j$). 
So $G'-\alpha$ has an $n\times m$ grid subgraph $J$ with vertex-set 
$\{1,\dots,n\}\times\{1,\dots,m\}$. 

For $j\in\{1,\dots,m\}$, let $R_j:= \{ (i,j) \in V(J) : 1\leq i\leq n\}$ be the $j$-th row in $J$. 
For $i\in\{1,\dots,n\}$, let $C_i:= \{ (i,j) \in V(J) : 1\leq j\leq m\}$ be the $i$-th column in $J$. 

Since contractions do not increase distances, $G'$ has radius at most $r$ with centre $\alpha$. 
So for each vertex $x\in V(G')$ there is a path $P_x$ in $G'$ from $x$ to $\alpha$ with at most $r-1$ internal vertices. Let $I_x$ be the set of internal vertices of $P_x$, which may be empty if $x\alpha\in E(G')$.

Let $s:=2r-1$. We may assume that $m>2s$ and $n>2s$, otherwise the claim is trivial. Let $p:= \floor{(n-2r+1)/2r)}$, so $p\geq 1$ and $n \geq 2rp+(2r-1)$.

Let $A:=\{ (2ir,j) \in V(J): 1\leq i\leq p, s+1\leq j\leq m-s\}$. So $|A|=p(m-2s)$. 

For $i\in\{1,\dots,p+1\}$, 
let $a_i$ be a random element in $\{2(i-1)r+1,\dots,2ir-1\}$, where each element is chosen with probability $\frac{1}{2r-1}$. 
Let 
$X:=\bigcup\{C_{a_i}: \text{odd }i\in\{1,\dots,p+1\} \}$ and
$Y:=\bigcup\{C_{a_i}: \text{even }i\in\{1,\dots,p+1\} \}$.
So $X\cup Y$ is a set of pairwise disjoint columns in $J$, alternating between $X$ and $Y$, as illustrated in \cref{FindK3t}.

\begin{figure}[!t]
\centering
\includegraphics{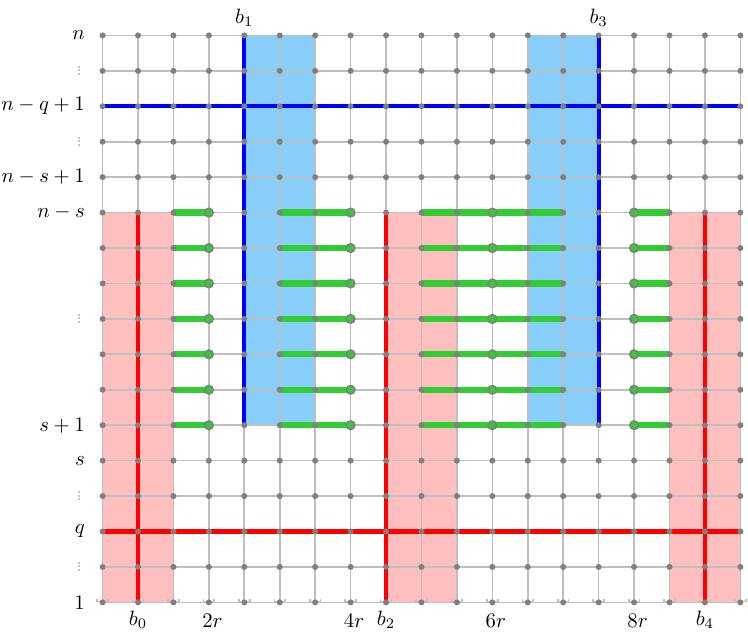}
\caption{Proof of \cref{K3tGridRadius}: the subgraphs $X'$ and $Y'$ are shown in red and blue; the vertical paths in $X'\cup Y'$ are chosen randomly from within each shaded block; and the vertices $x\in A'''$ are green, enlarged into the paths $Z_x$.}
\label{FindK3t}
\end{figure}

Let $A':= \{ x \in A: I_x \cap (X\cup Y)=\emptyset\}$. 
We now estimate $|A'|$. 
Consider $x\in A$.
Let $I_x$ be 1 if $x\in A$ and 0 otherwise. 
For each $y\in P_x$, the column containing $y$ is added to $X\cup Y$ with probability at most $\frac{1}{2r-1}$, so $\mathbb{P}(y\in X\cup Y)\leq\frac{1}{2r-1}$.
Since $|P_x|\leq r-1$, by the union bound, 
$\mathbb{P}(P_x \cap (X\cup Y)\neq\emptyset ) \leq\frac{r-1}{2r-1}<\frac12$.
So $\mathbb{P}(x \in A') >\frac12$, and
$\mathbb{E}(I_x)>\frac12$.
By definition, $|A'|=\sum_{x\in A}I_x$. 
By linearity of expectation, 
$\mathbb{E}(|A'|)=\sum_{x\in A}\mathbb{E}(I_x)>\sum_{x\in A}\frac12 = \frac12|A|$.
So there exists $a_1,\dots,a_{p+1}$ such that $|A'|>\frac12|A|$. 

For each $x\in A'$ let $Z_x$ be the maximal horizontal path in $J-(X\cup Y)$ including $x$. By the definition of $X$ and $Y$, the endpoints of $Z_x$ are adjacent to $X$ and $Y$. 

Consider the digraph $Q$ with $V(Q):=A'$, where $(x,y)\in E(Q)$ if $I_x\cap Z_y\neq\emptyset$. By construction, 
$Z_x\cap Z_y=\emptyset$ for distinct $x,y\in A'$. Thus $\deg^+_Q(x)\leq |I_x|\leq r-1$. Hence $|E(Q)|\leq(r-1)|V(Q)|$. Let $\overline{Q}$ be the undirected graph underlying $Q$. So $\overline{Q}$ has average degree $2|E(\overline{Q})|/|V(\overline{Q})|\leq 2|E(Q)|/|V(Q)|\leq 2(r-1)$. By Tur\'an's Theorem, $\overline{Q}$ has an independent set $A''$ with $|A''|\geq \frac{1}{2r-1}|A'|\geq 
\frac{1}{4r-2}|A|$. For distinct $x,y\in A''$, we have  
$I_x\cap Z_y=\emptyset$ and $I_y\cap Z_x=\emptyset$.

For each $j\in\{1,\dots,s\}$, 
let $L_j$ be the set of vertices $x\in A''$ such that 
$I_x\cap (R_j\cup R_{n-j+1})\neq\emptyset$. 
Since $|I_x|\leq r-1$, we have 
$\sum_{j=1}^s|L_j|\leq (r-1)|A''|$. 
Thus, 
$|L_q|\leq \frac{r-1}{s} |A''|
\leq \frac{1}{2} |A''|$ for some $q\in\{1,\dots,s\}$. 
Let $A''' := A''\setminus L_q$. Thus 
\begin{align*}
    |A'''| 
\geq \frac12 |A''| 
\geq \frac{|A|}{8r-4}
= \frac{p(m-2s)}{8r-4} 
& = \floor*{\frac{n-2r+1}{2r}} 
\left(\frac{m-4r+2)}{8r-4} \right) \\
& \geq \frac{(n-4r+2)(m-4r+2)}{8r(2r-1)}.
\end{align*}

Let $X':= X \setminus (R_{n-s+1}\cup \dots \cup R_n)\cup R_q$ and $Y':= Y \setminus (R_1 \cup \dots \cup R_s) \cup R_{n-q+1}$. Observe that $J[X']$ and $J[Y']$ are connected and disjoint. 
Let $R:=\{\alpha\}\cup\bigcup\{ P_x: x\in V(A''') \}$. 
So $G'[R]$ is connected, and $X,Y,R,\bigcup\{Z_x:x\in A'''\}$ are pairwise disjoint. 
Each $x\in A'''$ has a neighbour in $R$, and $Z_x$ is adjacent to both $X$ and $Y$. 
Thus $G'[R],G'[X'],G'[Y']$ and $\{Z_x:x\in A'''\}$ form a $K_{3,|A'''|}$ model in $G'$. 
\end{proof}

\cref{K3tGridRadius} leads to a proof of the following result from \cref{Intro}.

\blah*

\begin{proof}
Suppose for contradiction that there is an $n\times n$ grid minor in $G$. 
Let $\alpha$ be a centre of $G$.
So $G-\alpha$ has an $(n-1)\times (n-1)$ grid minor. 
By \cref{K3tGridRadius}, $G$ has a $K_{3,s}$ minor with 
\begin{align*}
    s
  \geq 
 \frac{(n-4r+1)^2}{8r(2r-1)}
\geq\frac{ ( n-4r+1 )^2}{16r^2} 
\end{align*}
Thus
$ ( n-4r+1 )^2 \leq  16r^2s \leq  16r^2(t-1)$, 
implying
$ n \leq  4r\sqrt{t-1}+4r-1 < 4r( 1+\sqrt{t-1})$, a contradiction.
\end{proof}

\section{Lower Bounds}

The next lemma leads to a proof of \cref{LowerBoundIntro}, which shows that our upper bounds in \cref{SurfaceIntro,K3tMinorFreeIntro} are best possible up to a constant factor. 

\begin{lem}
\label{LowerBound}
For any integers $r,k\geq 1$ there is a graph $G$ with radius at most $r$ and Euler genus at most $2k^2$, such that the $(2r-1)k \times (2r-1)k$ grid is a subgraph of $G$.
\end{lem}

\begin{proof}
As illustrated in \cref{LowerBoundK3t}, initialise $J$ to be the $(2r-1)k\times (2r-1)k$ grid with vertex-set $\{1,\dots,(2r-1)k\}^2$.
Let $W:= \{ ((2r-1)x-(r-1),(2r-1)y-(r-1)) : x,y\in\{1,\dots,k\}\}$. 
Add edges to $J$ so that every vertex in $J$ is at distance at most $r-1$ from $W$ while maintaining planarity (see \cref{LowerBoundK3t}). 
Let $G$ be obtained from $J$ by adding one new vertex $\alpha$ adjacent to each vertex in $W$. 
Every vertex in $G$ is at distance at most $r$ from $\alpha$.
Embed $J$ and $\alpha$ in a sphere $\Sigma_0$.
Let $\Sigma$ be the surface obtained from $\Sigma_0$ by adding $k^2$ handles, so that each edge of $G$ incident to $\alpha$ can be embedded in $\Sigma$ on one handle. This is possible since $\deg_G(\alpha)=|W|=k^2$. The result follows,
since $\Sigma$ has Euler genus $2k^2$. 
\end{proof}

\begin{figure}[t]
\centering
\includegraphics[angle=270]{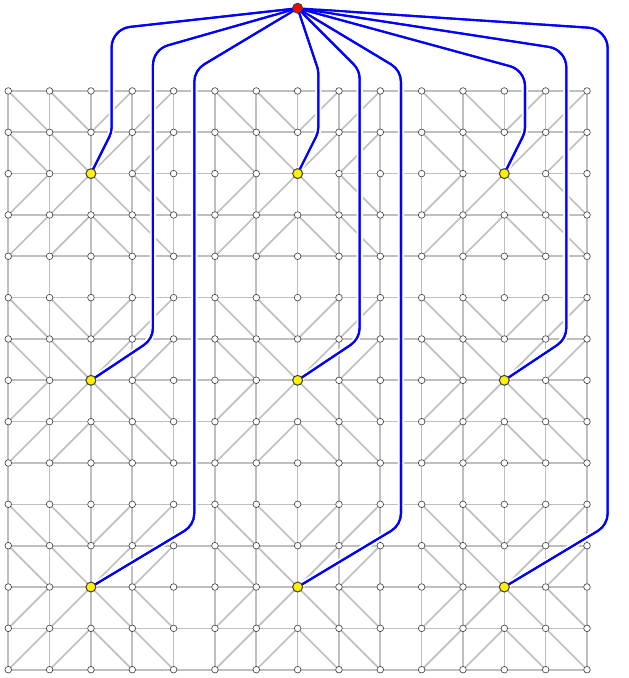}
\caption{Construction in the proof of \cref{LowerBound} with $r=k=3$, where vertices in $W$ are yellow.}
\label{LowerBoundK3t}
\end{figure}

To prove \cref{LowerBoundIntro}, given $g\geq 2$, apply \cref{LowerBound} with $k:=\floor{\sqrt{g/2}}$, implying $G$ has Euler genus at most $2k^2 \leq g$. 

The argument in \cref{LowerBound} generalises as follows. 

\begin{prop}
Let $H$ be a planar triangulation on $t\geq 4$ vertices. Let $A$ be the apex graph obtained from $H$ by adding a dominant vertex $z$. Then for any integer $r\geq1$ there is an $A$-minor-free graph $G$ with radius at most $r$, such that the $n \times n$ grid is a subgraph of $G$, where $n:=(2r-1)(\ceil{\sqrt{(t-3)/6}}-1) \approx r\sqrt{2t/3}$.
\end{prop}

\begin{proof}
By construction, $H$ has $3t-6$ edges, and $A$ has $n:=t+1$ vertices and $m:=4t-6$ edges. Say $A$ has Euler genus $g$. By Euler's formula, $4t-6=m\leq 3(n+g-2)= 3( t+1 + g-2)$, implying $g\geq \frac{t-3}{3}$. Let $G$ be the graph in the proof of \cref{LowerBound} with radius $r$ and $k=\ceil{\sqrt{(t-3)/6}}-1$. So the $(2r-1)k\times (2r-1)k$ grid is a subgraph of $G$. Moreover, $G$ has Euler genus at most $2k^2<\frac{t-3}{3}\leq g$, implying $G$ is $A$-minor-free.
 \end{proof}

\section{Tree-treewidth}\label{ttw}
	
Given a graph parameter $f$, \citet{LNW}  defined a new graph parameter $\text{tree-}f$ as follows. For a graph $G$, let \emph{$\text{tree-}f(G)$} be the minimum integer $k$ such that $G$ has a tree-decomposition $(B_x:x\in V(T))$ such that $f(G[B_x])\leqslant k$ for each node $x\in V(T)$. 
This concept in the case of $\text{tree-}\chi$ was introduced by \citet{Seymour16}. Tree-$\chi$, tree-$\alpha$ and tree-diameter have all recently been widely studied. 
\citet{LNW} introduced tree-treewidth. They showed that any proper minor-closed class has bounded tree-treewidth. The proof uses the Graph Minor Structure Theorem, so the bounds are large. It would be interesting to prove this without using the Graph Minor Structure Theorem and with better bounds. Here we do this for apex-minor-free graphs. Recall that for an apex graph $A$, $g(A,r)$ is the maximum treewidth of $A$-minor-free graphs with radius at most $r$. 

\begin{thm}
\label{TTW}
For any apex graph $A$ and any $A$-minor-free graph $G$, 
$$\ttw(G)\leq  g(A,2).$$ 
\end{thm}

\begin{proof}
We may assume that $G$ is connected. Let $u\in V(G)$ and $V_i:=\{w\in V(G): \dist_G(u,w)=i\}$ for $i\geq 0$. Observe that $(\{u\}\cup V_1\cup V_2,V_2\cup V_3,V_3\cup V_4,\dots)$ is a path-decomposition of $G$. We now bound the treewidth of each bag. Fix $i\geq 1$. Let $G_i$ be the graph obtained from $G[V_0\cup\dots\cup V_{i+1}]$ by contracting $G[V_0\cup \dots\cup V_{i-1}]$ into a vertex $u_i$. Note that $V(G_i)=\{u_i\}\cup V_i\cup V_{i+1}$, where every vertex in $V_i$ is adjacent to $u_i$ and every vertex in $V_{i+1}$ has a neighbour in $V_i$. Thus $G_i$ is $A$-minor-free with radius at most $2$, implying $\tw(G_i)\leq g(A,2)$. The $i$-th bag in the above path-decomposition induces a subgraph of $G_i$, and thus has treewidth at most $g(A,2)$. Hence $\ttw(G) \leq g(A,2)\}$.
\end{proof}


By \cref{TTW,PolyApexMinorTheorem}:

\begin{cor}
For any apex graph  $A$ with $t$ vertices, and for  every $A$-minor-free graph $G$,
$$\ttw(G)\in O^\ast(t^{18}).$$
\end{cor}

By \cref{TTW,K3tMinorFreeIntro}:

\begin{cor}
For every $K_{3,t}$-minor-free graph $G$,
$$\ttw(G) \in O^\ast( t^9 ).$$
\end{cor}

\section{Open Problems}

It is an open problem to determine tight bounds on the size of grid-minors in $A$-minor-free graphs with radius $r$, where $A$ is an apex graph with $t$ vertices. \cref{PolyApexMinorTheorem} provides an $O(rt^2)$ upper bound, and \cref{LowerBoundIntro} provides an $\Omega(r\sqrt{t})$ lower bound. There is also an $\Omega(t)$  lower bound  because it is known that if $H$ is the nested triangles graph with $t$ vertices and $H$ is a minor of the $n\times n$ grid, then $n\geq\Omega(t)$ (that is, \cref{PlanarGrid} is tight up to a constant factor). 

Our second open problem is inspired by the proof of  \cref{key}, which implicitly relies on the following fact: For any integers $p,q\geq 1$ there is a minimum integer $\ell(p,q)$ such that for any graph $G$ with treewidth at least $\ell(p,q)$:
\begin{enumerate}[(a)]
\item there are pairwise disjoint connected subgraphs $G_1,G_2,\dots,G_n$ such that $\tw(G_i)\geq p$ for each $i\in\{1,\dots,n\}$, and 
\item contracting each of $G_1,\dots,G_n$ to a vertex gives a minor of $G$ with treewidth at least $q$. 
\end{enumerate}
We claim that $\ell(p,q)\leq g( pq)$. Say  $\tw(G)\geq g(pq)$. 
So the $pq\times pq$ grid is a minor of $G$. 
Partition this grid into $q^2$ copies of the $p\times p$ grid. 
Each subgrid has treewidth $p$ and contracting each subgrid gives a $q\times q$ grid, which has treewidth $q$. By \eqref{GMT},  $\ell(p,q)\leq O^\ast( (pq)^9 )$. 
Are there better bounds on $\ell$? 
We propose the following.
\begin{conj}\label{con1}
$\ell(p,q)\leq O^\ast(pq)$.
\end{conj}
For example, $K_{pq}$ has $q$ pairwise disjoint copies of $K_p$ (each with treewidth $p-1$), such that contracting each copy gives $K_q$ (with treewidth $q-1$). \citet{CC13} have results about partitioning a graph with large treewidth into subgraphs with large treewidth (thus satisfying (a)), but without property (b). 
Indeed, \cref{con1} would imply and strengthen Conjecture 1 in \cite{CC13}, which considers only the number of vertex-disjoint graphs of treewidth $p$.

\paragraph{Acknowledgement: }
Research of both authors is supported by the Australian Research Council. Partially completed while the second author was in residence at the \emph{Extremal Combinatorics} program of the Simons Laufer Mathematical Sciences Institute in Berkeley, California (supported by the NSF grant DMS-1928930). 
		
{\fontsize{10pt}{11pt}
\selectfont
\bibliographystyle{DavidNatbibStyle}
\bibliography{DavidBibliography}}

\newpage
\appendix
\end{document}